\documentclass{article}

\usepackage[OT1]{fontenc}
\usepackage{amsthm,amsmath,natbib}
\usepackage[colorlinks,citecolor=blue,urlcolor=blue]{hyperref}
\usepackage{amsfonts}
\usepackage{amssymb}

\numberwithin{equation}{section}
\theoremstyle{plain}
\newtheorem{thm}{Theorem}[section]
\theoremstyle{remark}

\newtheorem{Lemma}{Lemma}[section]
\newtheorem{definition}{Definition}[section]

\begin{document}

\title{Benign overfitting without concentration}
\maketitle
\centerline{\author{Zong Shang\footnote{shangzong2117@mails.jlu.edu.cn, College of Computer Science and Technology, Jilin University, China.}}}





\begin{abstract}
We obtain a sufficient condition for benign overfitting of linear regression problem. Our result does not rely on concentration argument but on small-ball assumption and thus can holds in heavy-tailed case. The basic idea is to establish a coordinate small-ball estimate in terms of effective rank so that we can calibrate the balance of epsilon-Net and exponential probability. Our result indicates that benign overfitting is not depending on concentration property of the input vector. Finally, we discuss potential difficulties for benign overfitting beyond linear model and a benign overfitting result without truncated effective rank.
\end{abstract}
%


\section{Introduction}

In recent years, there are tremendous interest in studying generalization property of statistical model when it interpolates the input data. The classical learning theory suggests that when the predictor fits input data perfectly, it will suffer from noise so that it will not generalize well. To overcome this problem, regularization and penalized learning procedures like LASSO are studied to weaken the effect of noise to avoid overfitting. However, some empirical experiments indicate that overfitting may perform well. Why can overfitting perform well? In what cases can overfitting perform well? These questions motivated a series work on this field.\par
The original motivation is from the Deep Learning community, who empirically revealed that overfitting Deep Neural Network can still generalize well, see \cite{zhang2016understanding}. This counter-intuitive phenomenon still appear for linear regression and kernel ridge regression, see \cite{belkin2018understand} and \cite{tsigler2020benign}. They believe that investigating benign overfitting phenomenon in linear regression case will benefit to the more complex Deep Neural Networks case.\par
The cornerstone work \cite{bartlett2019benign} presented a minimax bound of generalization error of overfitting linear regression. Their result is in terms of effective ranks, which measures the tail behavior of eigenvalues of covariance matrix and will be defined later in our paper. Recently, \cite{chinot2020benign} improved their results to the large deviation regime. Both their results rely on the assumption that the input vector is a gaussian vector. This assumption is relaxed in \cite{tsigler2020benign} to sub-gaussian vector. However, all these work use concentration-argument and thus only adapt to a number of well-behaved distributions. In heavy-tailed case, small-ball method is often employed to study generalization property of statistical models, see \cite{koltchinskii2015bounding}. However, the original coordinate small ball estimate cannot be directly used when the input vector is anisotropic, but anisotropicity is a necessary condition for benign overfitting. Fortunately, this issue can be solved by a simple modification. In this paper, we derive a sufficient condition for benign overfitting when the input is heavy-tailed by using small-ball method.\par

Benign overfitting phenomenon was firstly discovered by \cite{zhang2016understanding}, and a great deal of effort in it has been devoted to the investigating its reason. \cite{mei2019generalization}, \cite{hastie2019surprises} studied asymptotic generalization error in random feature setting. In fact, benign overfitting in linear regression is not equivalent to that of random feature because parameters in random features cannot be controlled to minimize empirical loss like that in linear regression(only parameters in second layer can be used to minimize empirical loss while others in first layer are randomized). However, their empirical results illustrates that linear regression and random features of shallow neural network share similar double-descent risk curve. \cite{bartlett2019benign} obtained a two-sides non-asymptotic generalization bound of prediction error in gaussian linear regression setting. Their result is in terms of effective ranks, which is a truncated version of stable rank in Asymptotic Geometric Analysis. \cite{tsigler2020benign} generalized it into sub-gaussian linear regression. Their result is also in terms of effective ranks. \cite{liang2020just},\cite{rakhlin2019consistency}, and \cite{liang2020risk} studied benign overfitting in Reproduced Kernel Hilbert Space. \cite{belkin2019data} showed that interpolant maybe the optimal predictor in some cases.\par
The closest to our work is \cite{chinot2020benign}, who derived a sufficient condition for benign overfitting for gaussian linear regression in terms of effective rank. We obtain similar results but we only assume the input satisfies small ball assumption, in stead of gaussian distribution. Our result can also be partially compared to \cite{tsigler2020benign}, where they assumed both sub-gaussian and a marginal small ball assumption.

\subsection{Background and notation}

In this paper, we consider linear regression problems in $\mathbb{R}^p$. Given a dataset $D_{N}=\left\{\left(X_{i},Y_{i} \right ) \right\}_{i=1}^{N}$ and $Y_{i}=\left<X_{i},\alpha^{\ast} \right>+\xi_{i}$, where $\alpha^{\ast}\in\mathbb{R}^p$ is an unknown vector and $(X_{i})_{i=1}^{N}$ are i.i.d. copies of $X$, $\xi_{i}$ are unpredictable i.i.d. centered sub-gaussian noise, which is independent with $X$. Because we are going to compare linear regression with more general functions class later, we also often use $f_{\alpha}(\cdot)$ to denote $\left<\cdot,\alpha\right>$ in the following, and $\mathcal{F}_{A}$ as a set of $f_{\alpha}$ such that $\alpha\in A$. Assume the random vector $X\in\mathbb{R}^p$ satisfies weak small ball assumption with parameter $(\mathcal{L},\theta)$, which will be defined in Definition \ref{defi_wSBA} ,and denote its covariance matrix as $\Sigma$. Define the design matrix $\mathbf{X}$ with $N$ lines $X_{i}^{T}$. Denote response vector $\mathbf{Y}=\left(Y_{1},\cdots ,Y_{N}\right)$.\par
When $p>n$, the least-square estimator can interpolate $D_{N}$. Denote the one that has the smallest $\ell_{2}$ norm as $\hat{\alpha}$. That is to say,
\[
\hat{\alpha}=\mathbf{X}^{\dagger}\mathbf{Y},
\]
where $\mathbf{X}^{\dagger}$ is the Moore-Penrose pseudo inverse of $\mathbf{X}$. Denote $H_{N}\subset\mathbb{R}^p$ as $H_{N}=\left\{\alpha\in\mathbb{R}^p:\, \mathbf{X}\alpha=\mathbf{Y} \right\}$, we call $H_{N}$ as interpolation space. We have
\[
\hat{\alpha}=\underset{\alpha\in H_{N}}{\mathrm{argmin}}\left\|\alpha\right\|_{\ell_{2}}.
\]
We assume that $\mathrm{rank}(\Sigma)> N$, then $\hat{\alpha}$ exists almost surely.\par
Our loss function is squared loss, that is $\ell(t)=t^2$, and the loss of $\alpha$ is denoted by $\ell_{\alpha}=\ell\left(\left<\alpha-\alpha^{\ast},\mathbf{X} \right>\right)$. So the empirical excess risk is defined as
\[
\mathrm{P}_{N}{\mathcal{L}_{\alpha}}=\mathrm{P}_{N}\left(\ell_{\alpha}-\ell_{\alpha^{\ast}} \right).
\]

Benign overfitting depends on effective ranks of $\Sigma$. If $A\in\mathbb{R}^{p\times p}$ is a symmetric matrix, denote $\lambda_{1}(A)\geqslant\cdots\geqslant\lambda_{p}(A)$ as eigenvalues of $A$ and $s_{1}(A)\geqslant \cdots\geqslant s_{p}(A)$ be its singular values. If there is no ambiguity, we will write $s_{i}$ in stead of $s_{i}(A)$.\par
\cite{bartlett2019benign} defined two effective ranks:
\begin{eqnarray}\label{eq_effective_rank}
r_{k}(\Sigma)=\frac{\sum_{i>k}\lambda_{i}(\Sigma)}{\lambda_{k+1}},\quad R_{k}(\Sigma)=\frac{\left(\sum_{i>k}\lambda_{i}(\Sigma) \right)^2}{\sum_{i>k}\lambda_{i}^2 (\Sigma)}.
\end{eqnarray}
$R_{k}(\Sigma)$ is a truncated version of stable rank that occurred in Asymptotic Geometric Analysis, see \cite{vershynin2018high}, \cite{naor2017restricted} and \cite{mendelson2019stable} for a comprehensive review. Stable rank, denoted as $\mathrm{srank}_{q}(A)$, defined by
\[
\mathrm{srank}_{q}(A)=\left(\frac{\left\|A\right\|_{S_{2}}}{\left\|A\right\|_{S_{q}}} \right)^{\frac{2q}{q-2}},
\]
where $\left\|\cdot\right\|_{S_{q}}$ is the $q$-Schatten norm of $A$, that is to say, $\left\|A\right\|_{S_{q}}=\left(\sum_{i=1}^{p}s_{i}^q(A) \right)^{1/q}$. When $q=4$, and $A=\Sigma^{1/2}$, then
\[
\mathrm{srank}_{4}(\Sigma^{1/2})=\left(\frac{\left\|\Sigma^{1/2}\right\|_{S_{2}}}{\left\|\Sigma^{1/2}\right\|_{S_{4}}} \right)^4=\frac{\left(\sum_{i=1}^{p}s_{i}^2 (\Sigma^{1/2})\right )^2 }{\sum_{i=1}^{p}s_{i}^4 (\Sigma^{1/2})}=\frac{\left( \sum_{i=1}^{p}\lambda_{i} (\Sigma)\right )^2}{\sum_{i=1}^{p}\lambda_{i}^2 (\Sigma)}.
\]
It can be seen that $R_{k}(\Sigma)$ is the truncated version of $\mathrm{srank}_{4}(\Sigma^{1/2})$.\par
Apart from this, $r_{k}(\Sigma)$ is also a truncated version of the usual "effective rank" which is actually $\mathrm{tr}(\Sigma)/\lambda_{1}(\Sigma)$ in statistical literature, see \cite{koltchinskii2017concentration},\cite{rudelson2007sampling}.\par
In fact, our result will be in terms of $R_{k}(\Sigma)$ instead of $r_{k}(\Sigma)$, which is the usual choice in most past work. However, this does not matter, because the two effective ranks are closely related, we refer the reader to Appendix A.6 in \cite{bartlett2019benign} for a comprehensive review.\par 
For sake of simplicity, we define some extra notations. They have no special meanings, but will make our formula more clear. Denote
\begin{eqnarray}
R_{k,2}(\Sigma):=\left(1-\sqrt{\frac{k}{\mathrm{srank}_{4}(\Sigma)}} \right )R_{k}(\Sigma).
\end{eqnarray}
When $k=0$, we have $R_{k,2}(\Sigma)=\mathrm{srank}_{4}(\Sigma)$. Denote
\[
\mathfrak{R}_{k}(\Sigma):=\frac{(4p-k)^2}{8p}\frac{c^{\frac{16p^2}{(4p-k)^2}R_{k,2}(\Sigma)}}{R_{k,2}(\Sigma)},
\]
where $c<1$ is a constant.\par
Denote the operator norm of a matrix $A$ as $\left\|A\right\|$. Denote $\left\|\alpha\right\|_{A}$ as $\sqrt{\alpha^{T}A\alpha}$. We use $S(r)$ to denote the sphere in $\mathbb{R}^p$ with radius $r$ with respect to $\left\|\cdot\right\|_{\ell_{2}}$, $B(r)$ as ball analogously. Denote $S_{A}(r)$ and $B_{A}(r)$ as the corresponding sphere and ball with respect to $\left\|\cdot\right\|_{A}$. Denote $(\varepsilon_{i})_{i=1}^{N}$ are i.i.d. Bernoulli random variables. Denote $D$ as unit ball with respect to $L_{2}$ distance. If $F\subset L_{2}(\mu)$, let $\left\{G_{f}:\, f\in\mathcal{F}\right\}$ be the canonical gaussian process indexed by $F$, denote $\mathbb{E}\left \|G\right \|_{F}$ as
\[
\mathbb{E}\left \|G\right \|_{F}:=\mathrm{sup}\left \{\mathbb{E}\underset{f\in F'}{\mathrm{sup}}G_{f}:\, F'\subset F,\, F'\text{ is finite} \right \}.
\]
Denote
\[
\Lambda_{s_{0},u}(\mathcal{F})=\mathrm{inf}\underset{f\in\mathcal{F}}{\mathrm{sup}}\sum_{s\geqslant s_{0}}2^{s/2}\left\|f-\pi_{s}f\right\|_{(u^2 2^s)},\quad u\geqslant 1,\, s_{0}\geqslant 0.
\]
where the infimum is taken with respect to all admissible sequences $(F_{s})_{s\geqslant 0}$ and $\pi_{s}f$ is the nearest point in $F_{s}$ to $f$ with respect to $\left\|\cdot\right\|_{(u^2 2^s)}$. Here $\left\|\cdot\right\|_{(p)}=\mathrm{sup}_{1\leqslant q\leqslant p}\left\|\cdot\right\|_{L_{q}}/\sqrt{q}$. An admissible sequence is a sequence of partitions on $\mathcal{F}$ such that $\left|F_{s}\right|\leqslant 2^{2^s}$, and $\left|F_{0}\right|=1$, cf. \cite{mendelson2016upper}. Denote $d_{q}(F)$ as diameter of $F$ with respect to $\left\|\cdot\right\|_{L_{q}}$. Denote $\left\|\cdot\right\|_{\psi_{2}}$ as sub-gaussian norm. Denote $[N]$ as set $\left\{1,2,\cdots, N\right\}$. Denote $C, c, c_{0}, c_{1}, c_{2},\cdots$ as absolute constants.

\subsection{Structure of this paper}

Section \ref{section_preliminaries} contains some preliminaries knowledge. Section \ref{section_main} contains our main result, Theorem \ref{Theorem_main_result} and its proof are decomposed into two parts, which will be post-posed to section \ref{section_estimation_error} and section \ref{section_prediction_error}. These two sections contains estimation error and prediction error of interpolation procedure in linear regression case. In section \ref{section_discussion}, we will discuss why it is difficult to obtain oracle inequality beyond linear regression, and give a benign overfitting result without truncated effective rank by using Dvoretzky-Milman Theorem in Asymptotic Geometric Analysis.

\section{Preliminaries}\label{section_preliminaries}

In this section, we introduce some preliminary techniques which will be used to formulate and prove our main results. More precisely, we will introduce localization method to yield oracle inequality and small ball method to provide a lower bound of smallest singular value of design matrix.

\subsection{Localization Method}\label{subsection_generalization_error_interpolating_erm}
To get an oracle inequality, there are two approaches in general. The first one is called Isomorphism Method, which uses the isomorphy between between empirical and actual structures to derive an oracle inequality. The other is called localization method. In this work, we will use localization method to derive an oracle inequality.\par
For a statistical model $\mathcal{F}$ and $f^{\ast}\in\mathcal{F}$ is an oracle. Localization Method uses a $L_{2}$-ball centered at $f^{\ast}$ with radius $r$ to localize model $\mathcal{F}$. This allows us to study statistical properties of a learning procedure $\hat{f}$ on this small ball\footnote{This diameter is not necessarily to be of $L_{2}$ distance in localization method, though our choice is $L_{2}$ distance. We refer the interested reader to \cite{chinot2020statistical}}. More precisely, the radius $r$ captures upper bound estimation error $\left\|f-f^{\ast}\right\|_{L_{2}}$ for all $f\in \mathcal{F}\cap rD$. Therefore, if we can find an upper bound of $r$, we find the estimation error of learning procedure $\hat{f}$. Analog to that in \cite{chinot2020benign}, our localized set in this paper is
\[
\mathcal{F}_{H_{r,\rho}}=\left\{\left<\cdot,\alpha\right>\, :\alpha\in H_{r,\rho} \right\},\quad\text{where }H_{r,\rho}:=B(\rho)\cap B_{\Sigma}(r),
\]
where $\rho$ is upper bound of estimation error, which will be studied in section \ref{section_estimation_error}, and $r$ is upper bound of prediction error, which will be studied in section \ref{section_prediction_error}. Obtaining prediction risk is based on estimation risk. In this paper, we obtain estimation risk by studying minimum $\ell_{2}$ interpolation procedure and obtain prediction error by localization method based on it.\par
Optimal level of $r$, denoted as $r^{\ast}$ is carefully chosen by fixed points called complexity parameters.

\subsubsection{Complexity parameters}

In classical statistical learning theory, there are two common-used complexity parameters called multiplier complexity $r_{M}$ and quadratic complexity $r_{Q}$, we refer the reader to \cite{mendelson2016learning} for a comprehensive view. Quadratic complexity $r_{Q}$ is defined as follows:
\[
r_{Q,1}(\mathcal{F},\zeta)=\underset{r>0}{\mathrm{arginf}}\,\mathbb{E}\left\|G\right\|_{(\mathcal{F}-\mathcal{F})\cap rD}\leqslant \zeta r\sqrt{N},
\]
and
\[
r_{Q,2}(\mathcal{F},\zeta)=\underset{r>0}{\mathrm{arginf}}\,\mathbb{E}\underset{w\in (\mathcal{F}-\mathcal{F})\cap rD}{\mathrm{sup}}\left|\frac{1}{\sqrt{N}}\sum_{i=1}^{N}\varepsilon_{i}w(X_{i})\right|\leqslant \zeta r\sqrt{N},
\]
where $\zeta$ is an absolute constant.\par
$r_{Q}(\zeta_{1},\zeta_{2})=\mathrm{max}\left\{r_{Q,1}(\mathcal{F},\zeta_{1}),r_{Q,2}(\mathcal{F},\zeta_{2})\right\}$ is an intrinsic parameter. That is to say, $r_{Q}$ does not rely on noise $\xi$, but only on $\mathcal{F}$. This parameter measures the ability of $\mathcal{F}$ to estimate target function $f^{\ast}$.\par
While multiplier complexity $r_{M}$ is defined as follows:
\[
\phi_{N}(r):=\underset{w\in (\mathcal{F}-\mathcal{F})\cap rD}{\mathrm{sup}}\frac{1}{\sqrt{N}}\sum_{i=1}^{N}\varepsilon_{i}\xi_{i}w(X_{i}),
\]
\[
r_{M,1}(\kappa,\delta)=\underset{r>0}{\mathrm{inf}}\left\{\mathbb{P}\left\{\phi_{N}(r)\leqslant r^2 \kappa\sqrt{N} \right\}\geqslant 1-\delta\right\},
\]
and
\[
r_{0}(\kappa)=\underset{r>0}{\mathrm{inf}}\underset{w\in (\mathcal{F}-\mathcal{F})\cap rD}{\mathrm{sup}}\left\{{\left\|\xi w(X)\right\|_{L_{2}}\leqslant\frac{1}{2}\sqrt{N}\kappa r^2}\right\},
\]
where $\kappa$ is an absolute constant.\par
Then $r_{M}(\kappa,\delta)=r_{M,1}+r_{0}$ is called multiplier complexity, which measures the interplay between noise $\xi$ and function class $\mathcal{F}$.\par
Classical learning theory employs $r_{M}$ to measure the ability of $\mathcal{F}$ to absorb noise $\xi$. However, this parameter does not make sense in interpolation case. This is because interpolant $\hat{f}$ causes no loss on $r_{M}$ by interpolating $(X_{i},Y_{i})_{i=1}^{N}$ perfectly. Therefore, $r_{M}=0$ in this case. However, since $\hat{f}$ interpolates $(X_{i},Y_{i})_{i=1}^{N}$, it bears influence from noise $\xi$ so that $r_{Q}$ is no longer an intrinsic parameter. That is to say, $r_{Q}$ relies on $\xi$ implicitly because $\hat{f}$ has to estimate both signal and noise. It is this that causes the biggest difference from interpolation case and classical learning theory. Therefore, our complexity parameter is a variant of quadratic complexity, which will be defined in Equation \ref{eq_r_ast}.\par
Localization method employed complexity parameters to provide radius of localized set. However, we need to illustrate that interpolant $\hat{f}$ lies in it with high probability. This step is guaranteed by an exclusion argument.

\subsubsection{Exclusion}

For all $f\in\mathcal{F}$, if $f$ wants to be an interpolant, its empirical excess risk must be lower than a fixed level with high probability. To see this, we first decompose $\mathrm{inf}_{f\in\mathcal{F}}\mathrm{P}_{N}\mathcal{L}_{f}$ to its lower bound.\par
There are two decompositions of empirical excess risk into quadratic and multiplier components. The first one is as follows:
\begin{eqnarray}
\underset{f\in\mathcal{F}}{\mathrm{inf}}\,\mathrm{P}_{N}\mathcal{L}_{f}&=&\underset{f\in\mathcal{F}}{\mathrm{inf}}\frac{1}{N}\sum_{i=1}^{N}(f(X_{i})-Y_{i})^2-(f^{\ast}(X_{i})-Y_{i})^2\nonumber\\
&\geqslant&\underset{f\in\mathcal{F}}{\mathrm{inf}}\left \{\frac{1}{N}\sum_{i=1}^{N}(f-f^{\ast})^2(X_{i}) \right \}-2\underset{f\in\mathcal{F}}{\mathrm{sup}}\left \{\frac{1}{N}\sum_{i=1}^{N}\xi_{i}(f-f^{\ast})(X_{i})\right \}\nonumber\\
&:=&\underset{f\in\mathcal{F}}{\mathrm{inf}}\,\mathrm{P}_{N}\mathcal{Q}_{f-f^{\ast}}-2\underset{f\in\mathcal{F}}{\mathrm{sup}}\,\mathrm{P}_{N}\mathcal{M}_{f-f^{\ast}}.\label{eq_decomposition_excess_risk_first}
\end{eqnarray}
This kind of decomposition needs lower bound of quadratic component $\mathrm{P}_{N}\mathcal{Q}_{f-f^{\ast}}$. This lower bound is provided by small ball method. We will use this approach in subsection \ref{subsection_without_effective_rank} to acquire a sufficient condition for benign overfitting without truncated effective rank.
We turn to the second kind of decomposition. Our localized statistical model $\mathcal{F}=\mathcal{F}_{H_{r,\rho}}$ is a class of linear functionals on $\mathbb{R}^p$. Recall that the optimal choice of $r$ is denoted as $r^{\ast}$, so denote $\theta=r^{\ast}/r\in (0,1)$, and $\alpha_{0}=\alpha^{\ast}+\theta (\alpha-\alpha^{\ast})$, so $\left\|\Sigma^{1/2}(\alpha_{0}-\alpha^{\ast})\right\|_{\ell_{2}}=r^{\ast}$ and $\left\|\alpha_{0}-\alpha^{\ast}\right\|_{\ell_{2}}\leqslant \theta\rho$. Denote
\[
Q_{r,\rho}=\underset{\alpha-\alpha^{\ast}\in H_{r,\rho}}{\mathrm{sup}}\left|\frac{1}{N}\sum_{i=1}^{N}\left<X_{i},\alpha-\alpha^{\ast} \right>^2-\mathbb{E}\left<X_{i},\alpha-\alpha^{\ast}\right>^2 \right| ,
\]
and
\[
M_{r,\rho}=\underset{\alpha-\alpha^{\ast}\in H_{r,\rho}}{\mathrm{sup}}\left|\frac{2}{N}\sum_{i=1}^{N}\xi_{i}\left<X_{i},\alpha-\alpha^{\ast}\right> \right|.
\]
Then
\begin{eqnarray}
\underset{f\in\mathcal{F}}{\mathrm{inf}}\,\mathrm{P}_{N}\mathcal{L}_{f}&=&\underset{f\in\mathcal{F}}{\mathrm{inf}}\frac{1}{N}\sum_{i=1}^{N}(f(X_{i})-Y_{i})^2-(f^{\ast}(X_{i})-Y_{i})^2\nonumber\\
&\geqslant&\theta^{-2}\left((r^{\ast})^2- Q_{r^{\ast},\rho}\right )-2M_{r^{\ast},\rho}.\label{eq_decompose}
\end{eqnarray}

Suppose $(\xi_{i})_{i=1}^{N}$ are i.i.d. sub-gaussian random variables, then by Bernstein's inequality, with probability at least $1-\mathrm{exp}(-N/2)$ we have
\[
\frac{1}{N}\sum_{i=1}^{N}\xi_{i}^2\leqslant \frac{1}{2}\left \|\xi\right \|_{\psi_{2}}^2.
\]
Because interpolation procedure $\hat{f}$ interpolates all these inputs $(X_{i},Y_{i})_{i=1}^{N}$, the excess risk of $\hat{f}$ can be obtained by the noises,
\[
\mathrm{P}_{N}\mathcal{L}_{\hat{f}}=\mathrm{P}_{N}\left (\ell_{\hat{f}}-\ell_{f^{\ast}} \right )=-\mathrm{P}_{N}\ell_{f^{\ast}}=-\mathrm{P}_{N}\xi^2,
\]
then with probability at least $1-\mathrm{exp}(-N/2)$, we have
\[
\mathrm{P}_{N}\mathcal{L}_{\hat{f}}\leqslant -\frac{1}{2}\left \|\xi\right \|_{\psi_{2}}^2.
\]
That is to say, if $f$ wants to be an interpolant, it must satisfy this upper bound. Otherwise, $f$ will be excluded because it has little probability to be an interpolant.\par
This upper bound is different from the case of non-interpolation setting, where it is $0$. It is smaller than $0$ because the interpolation procedure is more restrict(in the sense of the interpolation space is smaller than version space) than non-interpolation procedure like ERM. The smaller upper bound can exclude more functions than non-interpolation procedure.\par

Therefore, we just need to upper bound multiplier and quadratic processes in Equation \ref{eq_decompose}, such that the lower bound of empirical excess risk over all $f\in \mathcal{F}_{H_{r,\rho}}$ is greater than $-\left\|\xi\right\|_{\psi_{2}}^2/2$ when $r$ is greater than a fixed level $r^{\ast}$. So that functions in $\mathcal{F}_{H_{r,\rho}}$ will be excluded from being an interpolant with high probability. Therefore, with high probability, interpolant will lie in $H_{r^{\ast},\rho}$ and we can upper bound its prediction risk by $r^{\ast}$.\par

\subsubsection{Multiplier process and Quadratic process}

We employ upper bounds of multiplier process and quadratic process in \cite{mendelson2016upper}:
\begin{Lemma}[Theorem 4.4 in \cite{mendelson2016upper}:Upper bound of multiplier process]\label{Lemma_upper_multiplier}
There exist absolute constants $c_{1}$ and $c_{2}$ for which the following holds. If $\xi\in L_{\psi_{2}}$ then for every $u,w\geqslant 8$, with probability at least $1-2\mathrm{exp}(-c_{1}u^2 2^{s_{0}})-2\mathrm{exp}(-c_{1}Nw^2)$,
\[
\underset{f\in\mathcal{F}}{\mathrm{sup}}\left |\frac{1}{\sqrt{N}}\sum_{i=1}^{N}\left (\xi_{i}f(X_{i})-\mathbb{E}\xi f \right ) \right |\leqslant cuw\left \|\xi\right \|_{\psi_{2}}\tilde{\Lambda}_{s_{0},u}(\mathcal{F}),
\]
where $c$ is an absolute constant.
\end{Lemma}
\begin{Lemma}[Theorem 1.13 in \cite{mendelson2016upper}: Upper bound of empirical process]\label{Lemma_quadratic_process}
There exists a constant $c(q)$ that depends only on $q$ for which the following holds. Then with probability at least $1-2\mathrm{exp}\left(-c_{0}u^2 2^{s_{0}} \right )$,
\[
\underset{f\in\mathcal{F}}{\mathrm{sup}}\left |\frac{1}{N}\sum_{i=1}^{N}\left (f^2(X_{i})-\mathbb{E}f^2 \right ) \right |\leqslant \frac{c(q)}{N}\left(u^2\tilde{\Lambda}_{s_{0},u}^2 (\mathcal{F})+u\sqrt{N}d_{q}(\mathcal{F})\tilde{\Lambda}_{s_{0},u}(\mathcal{F}) \right ),
\]
where $c(q)$ is a constant depending on $q$. Particularly, if $\mathcal{F}$ is a sub-gaussian class, then with probability at least $1-2\mathrm{exp}(-c_{0}N)$,
\[
\underset{f\in\mathcal{F}}{\mathrm{sup}}\left |\frac{1}{N}\sum_{i=1}^{N}\left (f^2(X_{i})-\mathbb{E}f^2 \right ) \right |\leqslant cL^2\mathbb{E}\left\|G\right\|_{\mathcal{F}}^2.
\]
\end{Lemma}

Set $2^{s_{0}}=k_{\mathcal{F}}$, where $k_{\mathcal{F}}=\left (\mathbb{E}\left\|G\right\|_{\mathcal{F}}/d_{2}(\mathcal{F}) \right )^2$ is the Dvoretzky-Milman Dimension of $\mathcal{F}$, we refer the reader to \cite{artstein-avidan2015asymptotic} for a comprehensive view. For $\ell_{2}^{p}$, the Dvoretzky-Milman dimension $k\sim p$, see e.g. Theorem 5.4.1 in \cite{artstein-avidan2015asymptotic}.
And $\tilde{\Lambda}(\mathcal{F})$ is called $\Lambda$-complexity of $\mathcal{F}$, which is a generalization of Gaussian complexity so that $\Lambda$-complexity just needs $\mathcal{F}$ has finite order of moments, instead of infinite order of moments. Particularly, when $\mathcal{F}$ happens to be a sub-gaussian class, $\tilde{\Lambda} (\mathcal{F})$ is equivalent to $\mathbb{E}\left\|G\right\|_{\mathcal{F}}$. We refer the reader to \cite{mendelson2016upper} for a comprehensive review.\par
In this paper, our function class is of linear functionals on $\mathbb{R}^p$, especially ellipses since we assume the random vector $X$ is not isotropic. Given the covariance matrix $\Sigma$ of random vector $X$. By Lemma 5 in \cite{chinot2020benign}, we obtain
\begin{eqnarray}\label{eq_gamma2_diameter_ellipsoid}
\mathbb{E}\left\|G\right\|_{\mathcal{F}_{H_{r,\rho}}}\leqslant \sqrt{2}\sqrt{\sum_{i=1}^{p} \mathrm{min}\left\{\lambda_{i}(\Sigma)\rho^2,r^2 \right \}}.
\end{eqnarray}

However, estimating $\tilde{\Lambda}(H_{r,\rho})$ is non-trivial unless $H_{r,\rho}$ is a sub-gaussian class. 
Note that the deviation in Lemma \ref{Lemma_quadratic_process} is neither optimal nor user-friendly(in the sense of the deviation parameter $u$ is coupled with complexity parameter $\mathbb{E}\left\|G\right\|_{\mathcal{F}}$). In fact, upper bound of quadratic process given by \cite{dirksen2015tail} is in optimal deviation when $\mathcal{F}$ is a sub-gaussian class. This is not fit to our heavy-tailed setup. It is non-trivial to obtain upper bound of quadratic process in heavy-tailed case. However, when $\mathcal{F}$ is sub-gaussian, we can omit parameter $r_{2}^{\ast}$ which will be defined in section \ref{section_main} and get a better bound of $\left\|\Gamma\right\|$ in subsection \ref{subsection_lower_bound_smallest_singular_value}, which will generate a preciser result, whereas the proof is omitted. To make our result uniform to both heavy-tailed and sub-gaussian case, we employ the one from \cite{mendelson2016upper} though it will not generate an optimal bound.

\subsection{Small Ball Method}

To deal with heavy-tailed case, we employ small ball method, which is a crucial argument in Asymptotic Geometric Analysis, see \cite{artstein-avidan2015asymptotic}. Small Ball Method in statistical learning theory is first developed in \cite{koltchinskii2015bounding}. It can be viewed as a kind of Paley-Zygmund method, which assumes the random vector is sufficiently spread, so that it will have many large coordinates. We refer the reader to \cite{mendelson2016learning} and \cite{mendelson2019stable} for a comprehensive view.\par
Classical small ball assumption is a lower bound on tail of random function, that is
\[
\mathbb{P}\left\{\left|X_{i}\right|\geqslant\kappa\left\|X\right\|_{L_{2}} \right\}\geqslant \theta,
\]
which can be verified by Paley-Zygmund inequality, see Lemma 3.1 in \cite{kallenberg2002foundations}, under $L_{4}-L_{2}$ norm equivalence condition. In this paper, small ball method is used to obtain lower bound of smallest singular value of design matrix. In statistical learning theory, small ball assumption is used to lead to lower bound of quadratic component in Equation \ref{eq_decomposition_excess_risk_first}, so that it can provide a lower bound of smallest singular value. As we do not need coordinates of input vector $X$ are independent, we need a small ball method without independent. Fortunately, independence assumption is relaxed in \cite{mendelson2019stable}, and the corresponding definition of small-ball assumption is as follows:
\begin{definition}[Small Ball Assumption:\cite{mendelson2019stable}]\label{defi_wSBA}
The random vector $X\in\mathbb{R}^p$ satisfies a weak small ball assumption(denoted as wSBA) with constants $\mathcal{L}, \kappa$ if for every $1\leqslant k\leqslant p-1$, every $k$ dimensional subspace $F$, every $z\in\mathbb{R}^p$,
\[
\mathbb{P}\left\{\left\|P_{F}X-z\right\|_{\ell_{2}}\leqslant\kappa\sqrt{k} \right\}\leqslant \left(\mathcal{L}\kappa\right)^k,
\]
where $P_{F}$ is the orthogonal projection onto the subspace $F$.
\end{definition}
There are many cases when random vector $X$ satisfying wSBA, we refer the reader to Appendix A in \cite{mendelson2019stable} for a comprehensive view.

\section{Main Result}\label{section_main}

In this section, we will formulate our main result, Theorem \ref{Theorem_main_result}. Before this, we have to assume our final assumption and define some parameters.\par
Firstly, We need to following assumption: There are constants $\delta_{1}>0$ and $\delta_{2}\geqslant 1$ such that
\begin{eqnarray}\label{eq_paley_zygmund}
\left (\frac{1}{p}\sum_{i=1}^{p}\left\|\Sigma^{1/2}e_{i}\right\|_{\ell_{2}}^{2+\delta_{1}} \right )^{\frac{1}{2+\delta_{1}}}\leqslant \delta_{2}\sqrt{\frac{\mathrm{tr}(\Sigma)}{p}},
\end{eqnarray}
where $(e_{i})_{i=1}^{p}$ are ONB of $\mathbb{R}^p$.\par
This assumption is used to select a proper(in sense of a uniform lower bound of inner product) subset $\sigma_{0}\subset [p]$ such that $\left|\sigma_{0}\right|\geqslant c_{0}p$, where $c_{0}$ depends on $\delta_{1}$ and $\delta_{2}$. This assumption is not restrictive. See subsection \ref{subsection_example} for an example.\par
Secondly, we define three parameters. Define $k^{\ast}$ as the smallest integer such that
\begin{eqnarray}\label{eq_k_ast}
p\log{\left (1+\frac{\sqrt{d_{q}(D) \frac{\tilde{\Lambda}(D)}{\sqrt{p}} +\frac{\tilde{\Lambda}^2(D)}{p}+\lambda_{1}(\Sigma)}\sqrt{\frac{p}{\mathrm{tr}(\Sigma)}}}{\sqrt{3c_{0}(1-\mathfrak{R}_{k}(\Sigma))-1}} \right )}
\leqslant N\frac{c_{0}\mathfrak{R}_{k}(\Sigma)+1-c_{0}}{2}+\log{p}
\end{eqnarray}
where the minimum of empty set is defined as $\infty$. Denote $\nu$ as follows:
\begin{eqnarray}\label{eq_nu}
\nu:=N\frac{c_{0}\mathfrak{R}_{k^{\ast}}(\Sigma)+1-c_{0}}{2}+\log{p}-p\log{\left (1+\frac{\sqrt{d_{q}(D) \frac{\tilde{\Lambda}(D)}{\sqrt{p}} +\frac{\tilde{\Lambda}^2(D)}{p}+\lambda_{1}(\Sigma)}\sqrt{\frac{p}{\mathrm{tr}(\Sigma)}}}{\sqrt{3c_{0}(1-\mathfrak{R}_{k^{\ast}}(\Sigma))-1}} \right )}.
\end{eqnarray}
Parameter $k^{\ast}$ is a level which can balance the two sides in Equation \ref{eq_k_ast}.\par
Denote
\begin{eqnarray}\label{eq_rho}
\rho=\left\|\alpha^{\ast}\right\|_{\ell_{2}}+\sqrt{\frac{2}{3(1-c_{0})c_{0}-1}}\sqrt{\frac{p}{\mathrm{tr}(\Sigma)}}\frac{\left\|\xi\right\|_{\psi_{2}}}{\varepsilon},
\end{eqnarray}
where $\varepsilon$ is a constant. $\rho$ will be upper bound of estimation error.\par
Denote
\[
r_{1}^{\ast}:=\underset{r>0}{\mathrm{arginf}}\left\{\tilde{\Lambda}(H_{r,\rho})\leqslant \sqrt{\zeta_{1} p}r  \right\},\quad r_{2}^{\ast}:=\underset{r>0}{\mathrm{arginf}}\left\{d_{q}(\mathcal{F}_{r,\rho})\leqslant \zeta_{2}r \right\}.
\]
\begin{eqnarray}\label{eq_r_ast}
r^{\ast}:=r_{1}^{\ast}+r_{2}^{\ast},
\end{eqnarray}
where $\zeta_{1},\zeta_{2}$ are absolute constants. Particularly, when $H_{r,\rho}$ is a sub-gaussian class, this definition reduces to that of \cite{chinot2020benign}. $r^{\ast}$ will be upper bound of prediction risk.\par
Now we can formulate our main result as follows.
\begin{thm}\label{Theorem_main_result}
Suppose $X=\Sigma^{1/2}Z\in\mathbb{R}^p$ is a random vector, where $Z$ is an isotropic random vector that satisfies wSBA with constants $\mathcal{L},\kappa$, and $\Sigma$ satisfies Equation \ref{eq_paley_zygmund}. If $(X_{i})_{i=1}^{N}$ are i.i.d. copies of $X$, forming rows of a random matrix $\mathbf{X}$. Let $\hat{\alpha}$ be an interpolation solution on $\left(X_{i},Y_{i} \right )_{i=1}^{N}$, where $Y_{i}=\left<X_{i},\alpha^{\ast}\right>+\xi_{i}$, and $(\xi_{i})_{i=1}^{N}$ are i.i.d. sub-gaussian random variables. Then there exists absolute constant $c$ such that: with probability at least $1-\mathrm{exp}(-\nu)-\mathrm{exp}(-cN)$,
\[
\left\|\hat{\alpha}-\alpha^{\ast}\right\|_{\ell_{2}}\leqslant \rho,\quad \left\|\Sigma^{1/2}\left(\hat{\alpha}-\alpha^{\ast}\right)\right\|_{\ell_{2}}\leqslant r^{\ast},
\]
where $\rho$, $\nu$ and $r^{\ast}$ are defined in Equation \ref{eq_rho}, \ref{eq_nu} and \ref{eq_r_ast}
\end{thm}

\subsection{Example}\label{subsection_example}

Consider a simple example considered in \cite{bartlett2019benign}, \cite{chinot2020benign} and \cite{tsigler2020benign}. When $X$ is a sub-gaussian random vector, we have $\tilde{\Lambda}(\mathcal{F})\sim\mathbb{E}\left\|G\right\|_{\mathcal{F}}$ at once. Therefore, with probability at least $1-\mathrm{exp}(-cN)$, we have
\[
\left\|\Gamma\right\|\lesssim \sqrt{N}\lambda_{1}(\Sigma).
\]
Consider a concrete case that $\varepsilon=o(1)$ such that for any $k$,
\[
\lambda_{k}(\Sigma)=\mathrm{e}^{-k}+\varepsilon,\quad \text{with}\quad \log{\frac{1}{\varepsilon}}<N,\quad p=cN\log{\frac{1}{\varepsilon}}.
\]
If $p\varepsilon =\omega (1)$, then $\mathrm{tr}(\Sigma)=O(1)$, and $d_{q}(D) \frac{\tilde{\Lambda}(\mathcal{F})}{\sqrt{p}} +\frac{\tilde{\Lambda}^2(\mathcal{F})}{p}+\lambda_{1}(\Sigma)=O(1)$. So
\[
\sqrt{\frac{p}{\mathrm{tr}(\Sigma)}}\sqrt{d_{q}(D) \frac{\tilde{\Lambda}(D)}{\sqrt{p}} +\frac{\tilde{\Lambda}^2(D)}{p}+\lambda_{1}(\Sigma)}=O(1).
\]
To choose $k$ such that Equation \ref{eq_k_ast} holds, we have to bound $\mathfrak{R}_{k}(\Sigma)$ from below. Firstly, we need to lower bound $R_{k}(\Sigma)$.
\[
R_{k}(\Sigma)=\Theta\left (\frac{\left(\mathrm{e}^{-k}+p\varepsilon \right )^2}{\mathrm{e}^{-2k}+p\varepsilon^2} \right )
\]
by setting $k=\log{(1/\varepsilon)}<\frac{p}{2}$, then $R_{k}(\Sigma)=\Theta (p)$.\par
Then we estimate $\mathrm{srank}_{4}(\Sigma)$. We have $\mathrm{srank}_{4}(\Sigma)=\Theta (p)$. So $R_{k,2}(\Sigma)=\Theta (p)$. Further, we have $\mathfrak{R}_{k}(\Sigma)=\Theta (2^{-p})$. Further, $2^{-p}=\Theta (\varepsilon^{cN})$.\par
Therefore,
\begin{eqnarray}
\mathrm{LHS}/p&=&\Theta\left(\log{\left(1+\frac{1}{\sqrt{3c_{0}(1-2^{-p})-1}}\right)} \right)\nonumber\\
&=&\Theta\left(-\log{\left(3c_{0}(1-\varepsilon^{cN})-1 \right)} \right)\nonumber\\
&=&\Theta\left(-\log{\left(1-\varepsilon^{cN} \right)} \right).\nonumber
\end{eqnarray}
For the right hand side of Equation \ref{eq_k_ast}, we have
\[
\frac{N}{p}\frac{c_{0}\mathfrak{R}_{k}(\Sigma)+1-c_{0}}{2}=\Theta\left(\frac{1-c_{0}(1-2^{-p})}{2c\log{(1/\varepsilon)}} \right)=\Theta\left(\varepsilon^{cN} \right),
\]
and
\[
\frac{\log{p}}{p}=\Theta\left(\frac{\log{c}+\log{N}+\log{\log{(1/\varepsilon)}}}{cN\log{(1/\varepsilon)}}\right)=\Theta\left(\frac{\log{N}}{N} \right).
\]
Recall that $\varepsilon=o(1)$, so Equation \ref{eq_k_ast} holds for $N$ large enough.\par
Consider $\delta_{2}$ such that $\delta_{2}^2\geqslant 1+1/\varepsilon$, then Equation \ref{eq_paley_zygmund} holds. Therefore, we can set
\[
\nu=\frac{\log{N}}{N}-\varepsilon^{cN}+\log{\left(1-\varepsilon^{cN} \right)},
\]
Next, we estimate $r^{\ast}$. Since $\mathcal{F}_{H_{r,\rho}}$ is a sub-gaussian class, $r_{2}^{\ast}=0$, and
\[
r_{1}^{\ast}=\underset{r>0}{\mathrm{arginf}}\left\{\sum_{i=1}^{p}\mathrm{min}\left\{r^2,\lambda_{i}(\Sigma)\rho^2 \right\}\leqslant \zeta_{1}pr^2 \right\}.
\]
So $r^{\ast}=r_{1}^{\ast}\leqslant \frac{2}{\sqrt{\zeta_{1}}}\left\|\alpha^{\ast}\right\|_{\ell_{2}}\sqrt{\frac{\mathrm{tr}(\Sigma)}{p}}$. Therefore, by Theorem \ref{Theorem_main_result}, with probability at least $1-\mathrm{exp}(-\nu)-\mathrm{exp}(-cN)$, we have
\begin{eqnarray}
\left\|\hat{\alpha}-\alpha^{\ast}\right\|_{\ell_{2}}&\leqslant& \left\|\alpha^{\ast}\right\|_{\ell_{2}}+C\left\|\xi\right\|_{\psi_{2}}\sqrt{\frac{p}{p\varepsilon+1}}\leqslant c\left\|\alpha^{\ast}\right\|_{\ell_{2}},\nonumber\\
 \left\|\Sigma^{1/2}(\hat{\alpha}-\alpha^{\ast})\right\|_{\ell_{2}}^2&\leqslant& \frac{4}{\zeta_{1}}\left\|\alpha^{\ast}\right\|_{\ell_{2}}^2\left( \frac{p\varepsilon+1}{p}\right )=\frac{4}{\zeta_{1}}\left\|\alpha^{\ast}\right\|_{\ell_{2}}^2\left(\varepsilon+\frac{1}{c\log{(1/\varepsilon)}N} \right ),\nonumber
\end{eqnarray}
if signal-to-noise ratio $\left\|\alpha^{\ast}\right\|_{\ell_{2}}/\left\|\xi\right\|_{\psi_{2}}$ is greater than $\sqrt{\frac{p}{p\varepsilon+1}}$.

\section{Estimation Error}\label{section_estimation_error}

In this section, we are going to obtain a upper bound of $\left\|\hat{\alpha}-\alpha^{\ast}\right\|_{\ell_{2}}$ in high probability. We have
\[
\hat{\alpha}=\mathbf{X}^{\dagger}\mathbf{Y}=\mathbf{X}^{\dagger}\mathbf{X}\alpha^{\ast}+\mathbf{X}^{\dagger}\xi.
\]
Therefore,
\begin{eqnarray}\label{eq_localized_estimation_error}
\left\|\hat{\alpha}-\alpha^{\ast}\right\|_{\ell_{2}}=\left\|\left(\mathbf{X}^{\dagger}\mathbf{X}-I\right)\alpha^{\ast}\right\|_{\ell_{2}}+\left\| \mathbf{X}^{\dagger}\xi\right\|_{\ell_{2}}\leqslant \left\|\alpha^{\ast}\right\|_{\ell_{2}}+\left\|\mathbf{X}^{\dagger}\right\|\left\|\xi\right\|_{\ell_{2}}.
\end{eqnarray}
For $\left\|\xi\right\|_{\ell_{2}}$, we can obtain
\[
\left\|\xi\right\|_{\ell_{2}}\leqslant \sqrt{N}\left\|\xi\right\|_{\psi_{2}},
\]
with probability at least $1-\mathrm{exp}(-N)$ by Bernstein's inequality.\par

To upper bound $\left\|\mathbf{X}^{\dagger}\right\|$, we need a lower bound of the smallest singular value of $\mathbf{X}$ in high probability.
\begin{Lemma}[Lower bound of the smallest singular value]\label{Lemma_smallest_singular_value}
Suppose $X=\Sigma^{1/2}Z\in\mathbb{R}^p$ is a random vector, where $Z$ is an isotropic random vector that satisfies wSBA with constants $\mathcal{L},\kappa$, and $\Sigma$ satisfies Equation \ref{eq_paley_zygmund}. If $(X_{i})_{i=1}^{N}$ are i.i.d. copies of $X$, forming rows of a random matrix $\mathbf{X}$. Then there exists constant $c_{0}$ such that the smallest singular value of $\mathbf{X}$ has lower bound
\[
s_{\mathrm{min}}(\mathbf{X})\geqslant \varepsilon\sqrt{\frac{3c_{0}(1-c_{0})-1}{2}}\cdot\sqrt{N}\sqrt{\frac{\mathrm{tr}(\Sigma)}{p}}\gtrsim\sqrt{N\frac{\mathrm{tr}(\Sigma)}{p}},\quad \forall \varepsilon\in (0,1)
\]
with probability at least $1-\mathrm{exp}\left (-\nu \right )-\mathrm{exp}(-cN)$, where $c$ is an absolute constant.
\end{Lemma}

With the help of Lemma \ref{Lemma_smallest_singular_value}, we can arrive at the estimation error:
\begin{thm}[Estimation Error]\label{theo_estimation_error}
Suppose $X=\Sigma^{1/2}Z$, where $Z$ is a random vector satisfying wSBA with parameters $\mathcal{L},\kappa$ and $\Sigma$ satisfies Equation \ref{eq_paley_zygmund}. Let $(X_{i})_{i=1}^{N}$ are i.i.d. copies of $X$. Let $Y_{i}=\left<X_{i},\alpha^{\ast}\right>+\xi_{i}$, where $(\xi_{i})_{i=1}^{N}$ are i.i.d. sub-gaussian random variables, and let $\mathbf{Y}=(Y_{i})_{i=1}^{N}$. Let $\mathbf{X}$ as random matrix with lines $X_{i}^{T}$, and $\hat{\alpha}=\mathbf{X}^{\dagger}\mathbf{Y}$. For $\nu$ defined in Equation \ref{eq_nu}, there exists constant $c_{0}$ such that: with probability at least $1-\mathrm{exp}(-cN)-\mathrm{exp}(-\nu)$,
\begin{eqnarray}
\left\|\hat{\alpha}-\alpha^{\ast}\right\|_{\ell_{2}}&\leqslant& \left\|\alpha^{\ast}\right\|_{\ell_{2}}+\sqrt{\frac{2}{3(1-c_{0})c_{0}-1}}\sqrt{\frac{p}{\mathrm{tr}(\Sigma)}}\frac{\left\|\xi\right\|_{\psi_{2}}}{\varepsilon},\quad \forall \varepsilon \in (0,1)\nonumber\\
&\leqslant&\left\|\alpha^{\ast}\right\|_{\ell_{2}}+c\left\|\xi\right\|_{\psi_{2}}\sqrt{\frac{p}{\mathrm{tr}(\Sigma)}}.\nonumber
\end{eqnarray}
\end{thm}
The proof is trivial by using Equation \ref{eq_localized_estimation_error} and Lemma \ref{Lemma_smallest_singular_value}.\par
Theorem \ref{theo_estimation_error} can be compared with Theorem 3 in \cite{chinot2020benign}. Their estimation error is related to effective rank $r_{k}(\Sigma)$, while our bound depends only on $\mathrm{tr}(\Sigma)/p$. This is because our lower bound on the smallest singular value is given by average eigenvalue, instead of effective rank. We believe that by choosing $c_{0}$, the smallest singular value can be controlled in terms of effective rank, though we think deriving such a bound in our work is not necessary.\par

In the following subsection, we are going to prove Lemma \ref{Lemma_smallest_singular_value}. An outline of the proof of Lemma \ref{Lemma_smallest_singular_value} is as follows. Firstly, we establish a coordinate small-ball estimation in terms of effective rank in Theorem \ref{theo_coordinate_small_ball}. Secondly, we prove a uniform lower bound of $\left\|\mathbf{X}t\right\|_{\ell_{2}}$ on an epsilon-Net of $S^{p-1}$. Finally, we can lower bound the smallest singular value by combining its minimal $\ell_{2}$ norm and its maximal operator norm.

\subsection{Coordinate small ball estimates in terms of effective rank}

In this subsection, we prove the following Theorem.
\begin{thm}[Coordinate Small Ball Estimate in terms of Effective ranks]\label{theo_coordinate_small_ball}
If random vector $X=\Sigma^{1/2}Z\in\mathbb{R}^p$ satisfies wSBA with constants $(\mathcal{L},\kappa)$, and $\left (e_{i}\right)_{i=1}^{p}$ are ONB of $\mathbb{R}^p$ that satisfy Equation \ref{eq_paley_zygmund}, then for $\varepsilon\in (0,1)$,
\begin{eqnarray}
\mathbb{P}\left\{\left | \left \{ i\leqslant p:\, \left |\left <\Sigma^{1/2}Z,e_{i} \right > \right |\geqslant \varepsilon\sqrt{\frac{\mathrm{tr}(\Sigma)}{p}} \right \} \right |\leqslant c_{0}p \right \}\lesssim \frac{(4p-k)^2}{8p}\frac{c^{\frac{16p^2}{(4p-k)^2}R_{k,2}(\Sigma)}}{R_{k,2}(\Sigma)},
\end{eqnarray}
where $c$ depends on $\mathcal{L},\kappa$.
\end{thm}

This is just a simple modification of that in \cite{mendelson2019stable}. We divide the proof of Theorem \ref{theo_coordinate_small_ball} into three steps.\par
Firstly, we select a proper subset $\sigma_{0}\subset [p]$. This step can be done by using a probabilistic combinatorics technique. Let $u_{i}$ be a random vector uniformly distributed on the given ONB $\left\{e_{i}\right\}_{i=1}^{p}$. Set indicators $(1_{i})_{i=1}^{p}$. If $\left\|\Sigma^{1/2}u_{i}\right\|_{\ell_{2}}\geqslant\sqrt{\mathrm{tr}(\Sigma)/(2p)}$, then $1_{i}=1$, otherwise $1_{i}=0$. Then
\[
\mathbb{E}\left [\sum_{i=1}^{p}1_{i} \right]=\sum_{i=1}^{p}\mathbb{P}\left\{\left\|\Sigma^{1/2}u_{i}\right\|_{\ell_{2}}\geqslant\sqrt{\frac{\mathrm{tr}(\Sigma)}{2p}}\right\}.
\]
Then by Equation \ref{eq_paley_zygmund} and Paley-Zygmund inequality, see e.g. Lemma 3.1 in \cite{kallenberg2002foundations}, we can get its lower bound: $\mathrm{RHS}\geqslant c_{0}(\delta_{1},\delta_{2})p$. Therefore, there exists a subset $\sigma_{0}\subset [p]$, whose cardinality is at least $c_{0}p$, such that for all $i\in\sigma_{0}$, there exists
\begin{eqnarray}\label{eq_lower_ell2_row}
\left\|\Sigma^{1/2}e_{i}\right\|_{\ell_{2}}\geqslant\sqrt{\frac{\mathrm{tr}(\Sigma)}{2p}}.
\end{eqnarray}
Secondly, \cite{mendelson2019stable} decompose $[c_{0}p]$ into $\ell$ coordinate blocks by using restricted invertibility Theorem. That is to say,
\begin{Lemma}[Lemma 3.1 in \cite{mendelson2019stable}]\label{Lemma_decompose}
Assume that for every $1\leqslant i\leqslant p$, Equation \ref{eq_paley_zygmund} holds. Set $k_{4}=\mathrm{srank}_{4}(\Sigma^{1/2})$. Then for any $\lambda\in (0,1)$, there are disjoint subsets $(\sigma_{i})_{i=1}^{\ell}\subset [c_{0}p]$ such that
\begin{itemize}
\item For $1\leqslant j\leqslant \ell$, there is $\left|\sigma_{j}\right|\geqslant c_{0}^4 k_{q}/2048$ and $\sum_{j=1}^{\ell}\left|\sigma_{j}\right|\geqslant c_{0}p/2$.
\item $\left\|\left((\Sigma^{1/2})^{\ast}P_{\sigma_{j}}^{\ast}\right)^{-1}\right\|_{S_{\infty}}\leqslant 4$.
\end{itemize}
\end{Lemma}

Next, we derive a uniform lower bound of $\left|\sigma_{j}\right|$ by lower bounding $\mathrm{srank}_{4}(\Sigma^{1/2})$.
\begin{Lemma}[Lower bound of stable rank in terms of effective rank]\label{Lemma_lower_stable_rank}
For $0\leqslant k\leqslant p-1$ and $\Sigma\in\mathbb{R}^{p\times p}$, we have
\[
\mathrm{srank}_{4}(\Sigma^{1/2})\geqslant \frac{16p^2}{(4p-k)^2}\left (1-\sqrt{\frac{k}{\mathrm{srank}_{4}(\Sigma)}} \right )R_{k}(\Sigma),
\]
\end{Lemma}
\begin{proof}
The proof is separated into two parts. Firstly, we lower bound $\left\|\Sigma^{1/2}\right\|_{S_{2}}$.\par
By Ky Fan's maximal principle, see e.g. Lemma 8.1.8 in \cite{stormer2013positive} or Chapter 3 in \cite{bhatia1997matrix}, we have
\[
\sum_{i=1}^{p-r}s_{i}^2 (\Sigma^{1/2})\geqslant \mathrm{tr}\left (\Sigma\left (I_{p}-P \right ) \right ),
\]
where $I_{p}-P$ is an orthogonal projection of rank $(p-r)$, which provides a lower bound of sum of largest $(p-r)$ eigenvalues of $\Sigma$. Set $p-r=k$, then $\mathrm{rank}(P)=p-k$. We have $\sum_{i>k}s_{i}^2 (\Sigma^{1/2})=\left\|\Sigma^{1/2}\right\|_{S_{2}}^2-\sum_{i=1}^{k}s_{i}^2 (\Sigma^{1/2})$. It follows that
\begin{eqnarray}\label{eq_upper_tail_eigenvalue_ky_fan}
\sum_{i>k}s_{i}^2 (\Sigma^{1/2})\leqslant \left\|\Sigma^{1/2}\right\|_{S_{2}}^2 -\mathrm{tr}\left (\Sigma\left (I_{p}-P \right ) \right ).
\end{eqnarray}
We just need to lower bound $\mathrm{tr}\left (\Sigma\left (I_{p}-P \right ) \right )$ in terms of $\left\|\Sigma^{1/2}\right\|_{S_{2}}^2$.
Consider $\Sigma $ and $\Sigma P$ separately, we have the following identity:
\[
\left\|P\Sigma^{1/2}\right\|_{S_{2}}^2 = \mathrm{tr}(\Sigma)-\mathrm{tr}\left (\Sigma\left (I_{p}-P \right ) \right ),
\]
so 
\begin{eqnarray}\label{eq_tr_sigma}
\mathrm{tr}\left (\Sigma\left (I_{p}-P \right ) \right )=\left\|\Sigma^{1/2}\right\|_{S_{2}}^2-\left\|P\Sigma^{1/2}\right\|_{S_{2}}^2.
\end{eqnarray}
Substitute Equation \ref{eq_tr_sigma} into Equation \ref{eq_upper_tail_eigenvalue_ky_fan}, then we just need to upper bound $\left\|P\Sigma^{1/2}\right\|_{S_{2}}^2$. However, by definition of $\left\|\cdot\right\|_{S_{2}}$ and property of Frobenius norm, we have 
\[
\left\|P\Sigma^{1/2}\right\|_{S_{2}}^2=\left\|\Sigma^{1/2}\right\|_{S_{2}}^2-\left\|P^{C}\Sigma^{1/2}\right\|_{S_{2}}^2,
\]
where $P^{C}$ is complement of projector $P$. Recall that $\mathrm{rank}(P)=p-k$, so we can set $P$ picking $p-k$ rows of $\Sigma^{1/2}$, so $P^{C}$ picks $k$ rows of $\Sigma^{1/2}$ and it has lower bound $\left\|\Sigma^{1/2}\right\|_{S_{2}}/(2\sqrt{p})$ by Equation \ref{eq_lower_ell2_row}. Therefore, we have
\[
\sum_{i>k}s_{i}^2 (\Sigma^{1/2})\leqslant \left (1-\frac{k}{4p} \right )\left\|\Sigma^{1/2}\right\|_{S_{2}}^2.
\]
and immediately,
\begin{eqnarray}\label{eq_upper_tail_singular_value}
\sum_{i=1}^{p}\lambda_{i}(\Sigma)\geqslant \left (1-\frac{k}{4p} \right )^{-1}\sum_{i>k}s_{i}^2 (\Sigma^{1/2})=\left (1-\frac{k}{4p} \right )^{-1}\sum_{i>k}\lambda_{i}(\Sigma).
\end{eqnarray}
Secondly, we upper bound $\left\|\Sigma^{1/2}\right\|_{S_{4}}$.\par
By Holder's inequality, 
\[
\left\|\Sigma\right\|_{S_{2}}^2=\sum_{i=1}^{k}s_{i}^2(\Sigma)+\sum_{i>k}s_{i}^2(\Sigma)\leqslant \sqrt{k}\left\|\Sigma\right\|_{S_{4}}^2+\sum_{i>k}s_{i}^2 (\Sigma).
\]
So we have
\begin{eqnarray}\label{eq_lower_tail_singular_value}
\sum_{i=1}^{p}\lambda_{i}^2(\Sigma)=\left\|\Sigma\right\|_{S_{2}}^2\leqslant \left (1-\sqrt{\frac{k}{\mathrm{srank}_{4}(\Sigma)}} \right )^{-1}\sum_{i>k}\lambda_{i}^2 (\Sigma),
\end{eqnarray}
by definition of stable rank.\par
Combining Equation \ref{eq_lower_tail_singular_value} and Equation \ref{eq_upper_tail_singular_value}, Lemma \ref{Lemma_lower_stable_rank} is proved.
\end{proof}

The rest of the proof is based on the following Lemma:
\begin{Lemma}[\cite{mendelson2019stable}]\label{Lemma_semi_finished}
If random vector $X$ satisfies the wSBA with constant $(\mathcal{L},\kappa)$, and $(e_{i})_{i=1}^{p}$ are ONB of $\mathbb{R}^{p}$, and $\Sigma^{1/2}:\mathbb{R}^p\to\mathbb{R}^p$ satisfies Equation \ref{eq_paley_zygmund}, then for $\varepsilon\in (0,1)$, we have
\[
\mathbb{P}\left\{\left | \left \{ i\leqslant p:\, \left |\left <\Sigma^{1/2}Z,e_{i} \right > \right |\geqslant \varepsilon\sqrt{\frac{\mathrm{tr}(\Sigma)}{p}} \right \} \right |\leqslant c_{0}p \right \}\leqslant \sum_{j\leqslant \ell}\left (\frac{e\mathcal{L}\varepsilon}{c_{0}/2} \right )^{\left|\sigma_{j}\right|/(c_{0}/2)},
\]
where $\sigma_{j}$, $c_{0}$, $\ell$ are the same as Lemma \ref{Lemma_decompose}.
\end{Lemma}
This Lemma is not explicitly given in \cite{mendelson2019stable}.\par
Using this, we can prove Theorem \ref{theo_coordinate_small_ball}:
\begin{proof}[Proof of Theorem \ref{theo_coordinate_small_ball}]
By Lemma \ref{Lemma_semi_finished}, Lemma \ref{Lemma_lower_stable_rank} and Lemma \ref{Lemma_decompose}, Theorem \ref{theo_coordinate_small_ball} is proved easily.
\end{proof}

\subsection{Lower bound of smallest singular value}\label{subsection_lower_bound_smallest_singular_value}

In this subsection, we proceed step 2 and 3. Build an epsilon-Net on $S^{p-1}$, obtain uniform lower bound of smallest singular value on it and extend it on the whole $S^{p-1}$.
\begin{proof}[Proof of Lemma \ref{Lemma_smallest_singular_value}]
Fix a random vector $X_{j}\in\mathbb{R}^p$. Consider $p$ unit vectors $(e_{i})_{i=1}^{p}$ which forming an ONB of $\mathbb{R}^{p}$, then by Theorem \ref{theo_coordinate_small_ball}, with probability at least $1-\mathfrak{R}_{k}(\Sigma)$, there exists a subset $\sigma_{j}$ with cardinality at least $c_{0}p$, and the vectors in it satisfy that for all $e_{i}\in \sigma_{j}$,
\begin{eqnarray}\label{eq_single_vector}
\left |\left < \Sigma^{1/2}Z_{j},e_{i} \right >\right |\geqslant \varepsilon\sqrt{\frac{\mathrm{tr}(\Sigma)}{p}}.
\end{eqnarray}
Each $X_{j}$ has such a subset $\sigma_{j}\subset (e_{i})_{i=1}^{p}$ of cardinality at least $c_{0}p$ with probability at least $1-\mathfrak{R}_{k}(\Sigma)$. Pick a $t\in (e_{i})_{i=1}^{p}$ randomly. If $t\in\sigma_{j}$, then Equation \ref{eq_single_vector} holds, that is to say, we can have a lower bound on the inner product at this time.\par
Denote $1_{j}$ as $1_{\left\{t\notin\sigma_{j}\right \}}$, then $\mathbb{E}1_{j}\leqslant 1-c_{0}(1-\mathfrak{R}_{k}(\Sigma))$. By Bernstein's inequality, with probability at least $1-\mathrm{exp}(-\mathrm{min}\left\{t^2, t \right\}N)$,
\[
\sum_{j=1}^{N}1_{j}\leqslant N\mathbb{E}1_{j}+t\leqslant \frac{3}{2}N(c_{0}\mathfrak{R}_{k}(\Sigma)+1-c_{0}).
\]
by setting $t=(c_{0}\mathfrak{R}_{k}(\Sigma)+1-c_{0})/2$, then with probability at least
\[
1-\mathrm{exp}\left (-N\left (c_{0}\mathfrak{R}_{k}(\Sigma)+1-c_{0}\right )/2 \right ),
\]
we have $\sum_{j=1}^{N}1_{j}\leqslant 3N(c_{0}\mathfrak{R}_{k}(\Sigma)+1-c_{0})/2$, that is to say,
\[
\sqrt{\sum_{j=1}^{N}\left<\Sigma^{1/2}Z_{j},u\right>^2}\geqslant \varepsilon\sqrt{\frac{3c_{0}(1-\mathfrak{R}_{k}(\Sigma))-1}{2}N}\sqrt{\frac{\mathrm{tr}(\Sigma)}{p}},\quad \forall u\in\sigma_{j}.
\]
Build an $\eta$-Net $\Gamma_{\varepsilon}$ on $B_{2}^p$.\par
Set
\[
\eta=\frac{\sqrt{\mathrm{tr}(\Sigma)}}{2\sqrt{2}\left\|\Gamma\right\|}\varepsilon\sqrt{3c_{0}(1-\mathfrak{R}_{k}(\Sigma))-1}\sqrt{\frac{N}{p}}
\]
By $\log{\left|\Gamma_{\eta}\right|}\leqslant p\log{\left (1+2/\eta\right)}$, we have
\[
\log{\left|\Gamma_{\eta}\right|}\leqslant p\log{\left (1+\frac{4\sqrt{2}}{\varepsilon}\frac{\left\|\Gamma\right\|}{\sqrt{\mathrm{tr}(\Sigma)}}\sqrt{\frac{p}{N}}\cdot\frac{1}{\sqrt{3c_{0}(1-\mathfrak{R}_{k}(\Sigma))-1}} \right )}
\]
We just need to ensure
\begin{eqnarray}
p\log{\left (1+\frac{\left\|\Gamma\right\|}{\sqrt{\mathrm{tr}(\Sigma)}}\sqrt{\frac{p}{N}}\cdot\frac{c}{\sqrt{3c_{0}(1-\mathfrak{R}_{k}(\Sigma))-1}} \right )}
\leqslant N\frac{c_{0}\mathfrak{R}_{k}(\Sigma)+1-c_{0}}{2}+\log{p}
\end{eqnarray}
by choosing $k$ wisely. Here we just need to repeat the selecting process no more than $\left \lceil \left|\Gamma_{\eta}\right|/(c_{0}p)\right \rceil$ times, so to make Equation \ref{eq_single_vector} holds uniformly for all elements in $\Gamma_{\eta}$, we have to pay a $\log{\left(\left \lceil \left|\Gamma_{\eta}\right|/(c_{0}p)\right \rceil\right)}$ in exponential term $\mathrm{exp}\left (-N\left (c_{0}\mathfrak{R}_{k}(\Sigma)+1-c_{0}\right )/2 \right )$ and make sure this probability is no greater than $1$.\par
For upper bound of $\left\|\Gamma\right\|$, we use Lemma \ref{Lemma_quadratic_process}. With probability at least $1-\mathrm{exp}(-c_{0}N)$,
\[
\left\|\Gamma\right\|=\sqrt{\underset{t\in S^{p-1}}{\mathrm{max}}\sum_{i=1}^{N}\left <\Gamma_{\cdot,i} ,t\right>^2_{\ell_{2}}}\leqslant \sqrt{N}\sqrt{C \left(d \frac{\tilde{\Lambda}(D)}{\sqrt{k_{D}}} +\frac{\tilde{\Lambda}^2(D)}{k_{D}}\right)+\lambda_{1}(\Sigma)}.
\]
By choosing $k=k^{\ast}$ defined in Equation \ref{eq_k_ast}, the probability can be lower bounded by $1-\mathrm{exp}(-\nu)$, where $\nu$ is defined in Equation \ref{eq_nu}.
In summary, with probability at least $1-\mathrm{exp}(-c_{0}N)-\mathrm{exp}(-\nu)$, we have
\[
s_{\mathrm{min}}(\mathbf{X})\geqslant \varepsilon\sqrt{\frac{3c_{0}(1-c_{0})-1}{8}N}\sqrt{\frac{\mathrm{tr}(\Sigma)}{p}}
\]
\end{proof}

\section{Prediction error}\label{section_prediction_error}

In this section, we obtain an upper bound of prediction error based on upper bound of estimation risk by using localization method introduced in subsetcion \ref{subsection_generalization_error_interpolating_erm}.
\begin{thm}[Prediction error]\label{Lemma_prediction_error}
If random vector $X=\Sigma^{1/2}Z\in\mathbb{R}^p$ satisfies wSBA with constants $(\mathcal{L},\kappa)$, and $\Sigma$ that satisfy Equation \ref{eq_paley_zygmund}. Then with probability at least $1-\mathrm{exp}\left(-\nu \right)-2\mathrm{exp}(-cN)$, prediction error satisfies
\[
\left\|\Sigma^{1/2}(\hat{\alpha}-\alpha^{\ast}) \right\|_{\ell_{2}}\leqslant r^{\ast},
\]
where $c$ is an absolute constant.
\end{thm}
The proof is a kind of localization argument. That is to say, we are going to prove $\hat{\alpha}-\alpha^{\ast}$ lies in a localized area with respect to $\left\|\cdot\right\|_{\Sigma}$. Firstly, we need a localization Lemma from \cite{chinot2020benign}:
\begin{Lemma}[Localization: Lemma 3 in \cite{chinot2020benign}]\label{Lemma_localization}
With probability at least $1-\mathrm{exp}(-N/16)$, we have $\mathrm{P}_{N}\mathcal{L}_{\hat{\alpha}}\leqslant -\left\|\xi\right\|_{\psi_{2}}^2/2$. Moreover, for any $r$, let $\Omega_{r,\rho}$ denote the following event
\[
\Omega_{r,\rho}=\left\{ \alpha\in\mathbb{R}^p:\, \alpha-\alpha^{\ast}\in B(\rho)\backslash B_{\Sigma}(r),\, \text{and }\mathrm{P}_{N}\mathcal{L}_{\alpha}>-\frac{1}{2}\left\|\xi\right\|_{\psi_{2}}^2 \right \}.
\]
On the event
\begin{eqnarray}\label{eq_localization}
\Omega_{r,\delta}\cap \left\{\hat{\alpha}-\alpha^{\ast}\in B(\rho)\right\}\cap\left\{\mathrm{P}_{N}\mathcal{L}_{\hat{\alpha}}\leqslant-\frac{1}{2}\left\|\xi\right\|_{\psi_{2}}^2\right\},
\end{eqnarray}
prediction risk has upper bound $r$, that is to say,
\[
\left\|\Sigma^{1/2}\left(\hat{\alpha}-\alpha^{\ast}\right)\right\|_{\ell_{2}}\leqslant r.
\]
\end{Lemma}
Lemma \ref{Lemma_localization} reduces upper bound of prediction risk to Equation \ref{eq_localization}. Probability of event $\left\{\hat{\alpha}-\alpha^{\ast}\in B(\rho)\right\}$ can be lower bounded by estimation error, see Lemma \ref{Lemma_smallest_singular_value}. Recall that $H_{r,\rho}=B(\rho)\cap B_{\Sigma}(r)$, we just need to prove that event
\[
\underset{\alpha\, :\alpha-\alpha^{\ast}\in H_{r,\rho}}{\mathrm{inf}}\mathrm{P}_{N}\mathcal{L}_{\alpha}>-\frac{1}{2}\left\|\xi\right\|_{\psi_{2}}^2
\]
holds with high probability for $r>r^{\ast}$.\par

In this section, we are going to find lower bound of $\mathrm{P}_{N}\mathcal{L}_{\alpha}$ in terms of upper bounds of quadratic and multiplier processes according to Equation \ref{eq_decompose}.\par
Firstly, we find upper bound of quadratic process. By Lemma \ref{Lemma_quadratic_process}, with probability at least $1-\mathrm{exp}(-cN)$, we have
\begin{eqnarray}\label{eq_upper_bound_quadratic}
Q_{r,\rho}\lesssim_{q}\left(d_{q}(H_{r,\rho})\frac{\tilde{\Lambda}(H_{r,\rho})}{\sqrt{p}}+\frac{\tilde{\Lambda}^2(H_{r,\rho}) }{p}\right).
\end{eqnarray}

Secondly, we need upper bound of multiplier component. This can be done by using Lemma \ref{Lemma_upper_multiplier}. That is to say, with probability at least $1-2\mathrm{exp}\left(-cN \right)$,
\begin{eqnarray}\label{eq_upper_bound_multiplier}
M_{r,\rho}\lesssim_{q} \left\|\xi\right\|_{\psi_{2}}\frac{\tilde{\Lambda}(H_{r,\rho})}{\sqrt{p}}.
\end{eqnarray}
since $\xi$ is centered and independent with $X$. Therefore, when $r>r^{\ast}$, we have
\[
d_{q}(H_{r,\rho})\leqslant \zeta_{1}r^{\ast}, \quad \frac{\tilde{\Lambda}(H_{r,\rho})}{\sqrt{p}}\leqslant \zeta_{2} r^{\ast}.
\]
\begin{proof}[Proof of Theorem \ref{Lemma_prediction_error}]
Let $\alpha\in\alpha^{\ast}+H_{r,\rho}$. Recall that $r=\left\|\Sigma^{1/2}(\alpha-\alpha^{\ast})\right\|_{\ell_{2}}$. If $r>r^{\ast}$, by substituting Equation \ref{eq_upper_bound_multiplier} and Equation \ref{eq_upper_bound_quadratic} into Equation \ref{eq_decompose}, it follows that
\[
\underset{\alpha\in \alpha^{\ast}+H_{r,\rho}}{\mathrm{inf}}\,\mathrm{P}_{N}\mathcal{L}_{\alpha}>(r^{\ast})^2\left(\theta^{-2}-\zeta_{1}\zeta_{2}\theta^{-2}-\zeta_{2}\theta^{-2}-\zeta_{2} \right )-\frac{1}{2}\left\|\xi\right\|_{\psi_{2}}^2.
\]
Set $\zeta_{1},\zeta_{2}$ small enough, $\mathrm{RHS}>-\frac{1}{2}\left\|\xi\right\|_{\psi_{2}}^2$. By Lemma \ref{eq_localization}, Theorem \ref{Lemma_prediction_error} is proved.
\end{proof}

\section{Discussion}\label{section_discussion}

In this section, we discuss two aspects. Firstly, we discuss why it is so difficult to investigate benign overfitting beyond linear model. Secondly, we discuss a benign overfitting case without truncated effective rank.

\subsection{Why linear model?}\label{subsubsection_alpha}

In this subsection, we imagine statistical model $\mathcal{F}$ is the affine hull of sub-classes $(\mathcal{F}_{j})_{j=1}^{d}$, that is to say, for all $f\in\mathcal{F}$, there exists $f_{j}$ in each $\mathcal{F}_{j}$ and $\alpha_{j}\in\mathbb{R}$ such that $f=\sum_{j=1}^{d}\alpha_{j}f_{j}$. Denote $\alpha=(\alpha_{1},\cdots,\alpha_{d})\in\mathbb{R}^d$. Of course this is not the problem that we deal with in this paper, but considering such a general case like this would be benefit to understand the role of $\alpha$ and the difficulty to generalize benign overfitting beyond linear model.\par
Even in this kind of simple "additive model" case, benign overfitting is much more difficult. Firstly, $\hat{f}$ interpolates $(X_{i},Y_{i})_{i=1}^{N}$, but $\hat{f}_{j}$ need not interpolates them. In fact, they may differs a lot, see $\hat{f}(x)=-x+x=\hat{f}_{1}+\hat{f}_{2}$ can interpolate $(10,0)$ but $\hat{f}_{1}(10)=-\hat{f}_{2}(10)$. It is a difficult task to derive oracle inequality by studying $\mathcal{F}_{j}$.\par

If we minimizing $\left\|\alpha\right\|_{\ell_{2}}$ analogs to linear case, the minimization of $\left\|\alpha\right\|_{\ell_{2}}$ given interpolation condition $\sum_{j=1}^{p}\alpha_{j}f(X_{i})=Y_{i}$ for all $i=1,2,\cdots,N$ can be solved by Moore-Penrose inverse.\par
Condition on $(f_{j})_{j=1}^{p}$. Denote matrix $\Gamma$ as
\[
\Gamma=\left [
\begin{matrix}
f_{1}(X_{1}) & f_{2}(X_{1}) & \cdots & f_{p}(X_{1})\\
f_{1}(X_{2}) & f_{2}(X_{2}) & \cdots & f_{p}(X_{2})\\
\vdots & \vdots & \ddots & \vdots\\
f_{1}(X_{N}) & \cdots & \cdots & f_{p}(X_{N}).
\end{matrix}
\right ]_{N\times p}
\]
Denote $\mathbf{f}^{\ast}$ as $\left (f^{\ast}(X_{1}),f^{\ast}(X_{2}),\cdots, f^{\ast}(X_{N}) \right )$ and $\mathbf{\xi}$ as $\left (\xi_{1},\xi_{2},\cdots, \xi_{N} \right )$. Then the interpolation condition is equivalent to
\[
\mathrm{minimizing}\, \left\|\alpha\right\|_{\ell_{2}},\quad \mathrm{s.t.}\, \Gamma\alpha=Y.
\]
We assume $\alpha$ satisfying interpolation condition always exists. Using Moore-Penrose inverse, we have
\[
\hat{\alpha}=\Gamma^{\dagger}Y=\Gamma^{\dagger}\mathbf{f}^{\ast}+\Gamma^{\dagger}\xi.
\]
Therefore, to establish upper bound of $\left\|\hat{\alpha}\right\|_{\ell_{2}}$, we need a lower bound of the smallest singular value of $\Gamma$.\par
However, as we see in Lemma \ref{Lemma_smallest_singular_value}, the smallest singular value increases when $N$ increases, causing $\left\|\Gamma^{\dagger}\mathbf{f}^{\ast}\right\|_{\ell_{2}}$ decreasing. This phenomenon is called "signal blood" in \cite{muthukumar2020harmless}, which means that the influence caused by signal $\mathbf{f}^{\ast}$ will decline so that minimizing $\left\|\hat{\alpha}\right\|_{\ell_{2}}$ cannot reflect properties true signal unless there are some unrealistic restrictions.\par
Therefore, $\mathbf{f}^{\ast}$ should balance $\Gamma^{\dagger}$ when $N$ increase to avoid signal blood. This can be done by linear regression, where $\Gamma=\mathbf{X}$. This illustrates that why we choose linear model.

\subsection{Benign overfitting without truncated effective rank}\label{subsection_without_effective_rank}

Now, we try to establish benign overfitting without truncated effective rank, but on stable rank $r_{0}(\Sigma)$, see Equation \ref{eq_effective_rank}. Recall that a linear model on $T\subset\mathbb{R}^p$ is $\mathcal{F}_{T}=\left\{\left<\cdot,t\right>:\, t\in T\right\}$. Let $\sigma=(X_{1},\cdots,X_{N})$, then the projection of $\mathcal{F}_{T}$ by using $\sigma$ is indeed a random linear transformation of $T$. That is to say, $P_{\sigma}\mathcal{F}_{T}=\mathbf{X}T$, where $P_{\sigma}(f_{t})=(\left<X_{i},t\right>)_{i=1}^{N}$. We need lower bound of smallest singular value of $\mathbf{X}$ to derive an upper bound of estimation error, and a lower bound of quadratic component in Equation \ref{eq_decomposition_excess_risk_first}. Fortunately, this can be done by Dvoretzky-Milman Theorem, see \cite{artstein-avidan2015asymptotic} or \cite{mendelson2016dvoretzky}. Dvoretzky-Milman Theorem can hold with rather heavy-tailed random vectors, but for the sake of simplicity, we assume $(g_{i})_{i=1}^{N}$ are i.i.d. gaussian random vectors in $\mathbb{R}^p$.
\begin{Lemma}[Dovoretzky-Milman(one-side)]
There exists absolute constants $c_{1},c_{2}$ such that: If $0<\delta<\frac{1}{2}$, and
\begin{eqnarray}\label{eq_dimension_dvoretzky_milman}
N\leqslant c_{1}\frac{\delta^2}{\log{(1/\delta)}}r_{0}(\Sigma),
\end{eqnarray}
and $\Gamma=\sum_{i=1}^{N}\left<g_{i},\cdot\right>e_{i}$, where $(e_{i})_{i=1}^{N}$ are ONB of $\mathbb{R}^N$. Then with probability at least $1-2\mathrm{exp}(-c_{2}r_{0}(\Sigma)\delta^4/\log{(1/\delta)})$,
\[
(1-\delta)\sqrt{\mathrm{tr}(\Sigma)}B_{2}^{N}\subset \Gamma\left(\Sigma^{1/2}B_{2}^p \right).
\]
\end{Lemma}
Take $\delta=1/4$ for example, we have $4\mathrm{tr}(\Sigma)B_{2}^{N}\subset\Gamma(\Sigma^{1/2}B_{2}^p)$, so
\[
s_{\mathrm{min}}(\Gamma)=\sqrt{\underset{t\in S^{p-1}}{\mathrm{min}}\left<g_{i},t \right>^2}\geqslant 4\sqrt{\mathrm{tr}(\Sigma)}
\]
holds with probability at least $1-2\mathrm{exp}(-c r_{0}(\Sigma))$. Therefore, with probability at least $1-2\mathrm{exp}(-c_{1} r_{0}(\Sigma))-2\mathrm{exp}(-c_{2}N)$,
\[
\left\|\hat{\alpha}-\alpha^{\ast}\right\|_{\ell_{2}}\leqslant \left\|\alpha^{\ast}\right\|_{\ell_{2}}+\left\|\xi\right\|_{\psi_{2}}\sqrt{\frac{N}{\mathrm{tr}(\Sigma)}}.
\]
As for the prediction risk, we have: when $r>r_{1}^{\ast}$,
\[
\underset{\alpha\in\alpha^{\ast}+H_{r,\rho}}{\mathrm{inf}}\mathrm{P}_{N}\mathcal{L}_{\alpha}\geqslant r^2\left(\frac{16\mathrm{tr}(\Sigma)}{N}-\frac{1}{2}\zeta_{1} \right )-\frac{1}{2}\left\|\xi\right\|_{\psi_{2}}^2>-\frac{1}{2}\left\|\xi\right\|_{\psi_{2}}^2.
\]
From here on, the proof is as the same as that of Theorem \ref{Lemma_prediction_error}, the details are omitted.\par
Note that $r_{0}(\Sigma)\leqslant p$, and $p=cN\log{(1/\varepsilon)}$, so we can set $c_{1}, \delta$ wisely to adapt to the example discussed in subsection \ref{subsection_example}.\par

In summary, although interpolation learning suffers from estimating both noise $\xi$ and sign $\alpha^{\ast}$, it still generalize well if the smallest singular value of $\mathbf{X}$ is large enough such that it can absorb the level of noise, $\sqrt{N}\left\|\xi\right\|_{\psi_{2}}$, see Equation \ref{eq_localized_estimation_error}. The smallest singular value is used to weaken influence of noise. To make the smallest singular value large enough, the number of samples should satisfy an upper bound that depends on covariance of the input vector. This threshold is used to balance the rate of exponential decay(acquired by concentration or small-ball argument) and metric entropy(given by net argument). Therefore, this threshold depends on dimension $p$, sample size $N$ and covariance $\Sigma$. If we fix relationship between $p$ and $N$(like the example in subsection \ref{subsection_example}), we need $\Sigma$ has a large trace(or at least heavy tail of eigenvalues), which is the key to benign overfitting. Note that in this interpretation, there is no restrictions on concentration properties of input vector $X$, but its small-ball property, that is to say, $X$ should fully spread on its margin. It is its spreading that can absorb noise $\xi$. It is this that make minimum $\ell_{2}$ linear interpolant fit into heavy-tailed case. Finally, we believe that our result could be easily modified to "Informative-Outlier" framework, cf. \cite{chinot2020statistical}, to obtain a result in a "robust flavor" both for CS community and statistics community.

\bibliographystyle{plain}
\end{document}